\documentclass[10pt,a4paper]{amsart}
\usepackage{amsmath,amssymb,verbatim}
\usepackage[left=3cm,right=3cm,top=3.3cm,bottom=3.5cm]{geometry}
\newtheorem{thm}{Theorem}

\newtheorem{lem}{Lemma}

\theoremstyle{definition}

\newtheorem{remark}{Remark}

\newcommand{\D}{\mathbb D}
\newcommand{\N}{\mathbb N}

\newcommand{\Un}{1\hskip -1.25mm 1}
\numberwithin{equation}{section}

\title[Commutants of quasihomogeneous Toeplitz operators]{Commutants of quasihomogeneous Toeplitz operators on the harmonic Bergman space}

\author{Issam Louhichi} 
\address{King Fahd University of Petroleum \& Minerals\\ Department of Mathematics \& Statistics\\
 Dhahran 31261, Saudi Arabia}
\email{issam@kfupm.edu.sa}

\author{Fanilo Randriamahaleo}
\address{Universit\'e de Bordeaux\\
UFR de Math\'ematiques et Informatique\\
351, Cours de la Lib\'eration\\
33405 Talence, France}
\email{Fanilo.Randriamahaleo@math.u-bordeaux1.fr\\frandriamahaleo@gmail.com}

\author{Lova Zakariasy}
\address{High Institute of Technology\\
Industrial Engineering Department\\201 Antsiranana, Madagascar}
\email{lova.zakariasy@ist-antsiranana.mg}

\keywords{Harmonic Bergman space, Toeplitz operator, Mellin transform}
\begin{document}
\begin{abstract}
One of the major goals in the theory of Toeplitz operators on the Bergman space over the unit disk $\mathbb{D}$ in the complex place $\mathbb{C}$ is to completely describe the commutant of a given Toeplitz operator, that is, the set of all Toeplitz operators that commute with it. Here we shall study the commutants of a certain class of quasihomogeneous Toeplitz operators.
\end{abstract}
\thanks{The second author is supported by \textit{Agence Universitaire de la Francophonie} via \textit{Horizons Francophones} Fundamental Sciences Program.}

\maketitle

\section{Introduction}
Let $L^2 = L^2(\D,dA)$ denote the space of all square integrable functions in the open unit disk $\D$  with respect to the normalized Lebesgue area measure $dA=rdrd\theta$, where $(r,\theta)$ are the polar coordinates. The harmonic Bergman space, denoted by $L^2_h$, is the set of all harmonic functions on $\D$ that are in  $L^2$. It is well known that $L^2_h$ is a closed subspace of $L^2$, hence it is a Hilbert space with an orthonormal basis $\{\sqrt{n+1}z^n\}_{n=0}^\infty \cup \{\sqrt{n+1}\overline{z}^n\}$. Let $Q$ be the orthogonal projection of $L^2$ onto $L^2_h$. A Toeplitz operator $T_f$ on $L^2_h$ with symbol $f$ is defined by 
\[T_f(u) = Q(fu), \ \text{ for } u \in L^2_h.\]
Clearly if the symbol $f$ is bounded then $T_f$ is bounded on $L^2_h$. However, dealing only with bounded symbols is restrictive, as an unbounded function might be a symbol of bounded Toeplitz operator. So if $f \in L^1(\D,dA)$, the Toeplitz operator $T_f$ is densely defined  on $L^2_h$ as an integral operator
$$T_f(u) = \int_\D f(w)u(w)R_z(w)dA(w), \ \text{ for } u \in \mathcal{P}(z,\bar{z}),$$
where $R_z(.)$ is the reproducing kernel of $L^2_h$ and $\mathcal{P}(z,\bar{z})$ is the set of polynomials in $z$ and $\bar{z}$. In this paper we consider the symbols in $L^1(\D,dA)$ for which the  Toeplitz operator defined as above can be extended to a bounded operator on $L^2_h$. 

A symbol $f$ is said to be \emph{quasihomogeneous} of degree $p$  if it can be written as $f(re^{i\theta}) = e^{ip\theta}\phi(r)$, where $\phi$ is a radial function on $\mathbb{D}$. In this case, the associated Toeplitz operator $T_f$ is also called quasihomogeneous Toeplitz operator of degree $p$. Analogous quasihomogeneous Toeplitz operators, defined on the analytic Bergman space, have been extensively studies over the last decade. See for example \cite{cr, lr, lr2, lry, lsz, lz1}. Recently in \cite{lz2}, Louhichi and Zakariasy  investigated products of quasihomogeneous Toeplitz operators on the harmonic Bergman space while generalizing the results of Dong and Zhou in \cite{dz}. This paper is a continuation of the work started in \cite{lz2}. Here we plan to study the commutativity of the product of two quasihomogeneous Toeplitz operators whose degrees are of opposite sign. 

\section{Preliminaries}

Our  main tool in studying quasihomogeneous Toeplitz operators is the Mellin transform. Let $f$ be a radial function in $L^1([0,1], rdr)$. The Mellin transform $\widehat{f}$ of $f$ is defined by

	\[ \widehat{f}(z)=\int_{0}^{1} f(r) r^{z-1}\,dr=\mathcal{M}(f)(z). \]

It is well known that, for these functions, the Mellin transform is well defined on the right half-plane $\{z : \Re z\geq 2\}$ and is analytic on $\{z:\Re z > 2\}$. We also define the Mellin convolution of two functions $f$ and $g$ in $L^1([0,1], rdr)$  by
\[ (f \ast_M g) (r) = \int_0^1 f(t) g(\frac{r}{t})\frac{dt}{t}. \]
It is easy to see that the Mellin transform converts the Mellin convolution product into a pointwise product, that is, 
 \[ \widehat{(f \ast_M g)}(z) = \widehat{f}(z)\widehat{g}(z). \]

Hereafter, the Mellin transform $\widehat{f}$ is uniquely determined by its values on any arithmetic sequence of integers, according to the following Lemma \cite[p.102]{rem}.
\begin{lem}
\label{blk1}
 Suppose that $f$ is a bounded analytic function on
$\{ z : \Re z>0\}$ which vanishes at the pairwise distinct points
$z_1, z_2 \cdots$, where
\begin{itemize}
\item[i)] $\inf\{|z_n|\}>0$\\
 and
\item[ii)] $\sum_{n\ge 1}\Re(\frac{1}{z_n})=\infty$.
\end{itemize}
Then $f$ vanishes identically on $\{ z : \Re z>0\}$.
\end{lem}

\begin{remark}
\label{blk2} Now one can apply this theorem to prove that if $f\in
L^1([0,1],rdr)$ and if there exist $n_0, p\in\mathbb{N}$ such that
$$\widehat{f}(pk+n_0)=0\textrm{ for all } k\in\mathbb{N},$$
then $\widehat{f}(z)=0$ for all $z\in\{z:\Re z>2\}$ and so $f=0$.
\end{remark}

We shall often use \cite[Lemma 2.5~p.1767]{dz} which can be stated as follows.
\begin{lem}\label{lemdz}
Let $p$ be an integer and $\phi$ be a bounded radial function. For each $k \in \N$,
\begin{align*}
		T_{e^{ip\theta}\phi}(z^k) = & \left\{ \begin{array}{ll}
	(2k+2p+2)\widehat{\phi}(2k+p+2)z^{k+p} & \text{ if }  k \geq -p\\
	(-2k-2p+2)\widehat{\phi}(-p+2)\overline{z}^{-k-p} & \text{ if } k < -p, 
\end{array}	\right.\\
		T_{e^{ip\theta}\phi}(\overline{z}^k) = & \left\{ \begin{array}{ll}
	(2k-2p+2)\widehat{\phi}(2k-p+2)\overline{z}^{k-p} & \text{ if }  k \geq p\\
	(2p-2k+2)\widehat{\phi}(p+2)z^{p-k} & \text{ if } k < p. 
\end{array}	\right.
\end{align*}
\end{lem}

The next lemma is a direct application of Lemma \ref{lemdz}.  
\begin{lem}\label{lm2}
Let $p$, $s$ be two positive integers and $\phi$, $\psi$ be radial functions. If 
$$T_{e^{ip\theta}\phi}T_{e^{-is\theta}\psi}=T_{e^{-is\theta}\psi}T_{e^{ip\theta}\phi},$$
then the following equalities hold :
\begin{itemize}
	\item For $ k \leq |s-p|$ :
\begin{align}
p \leq s : & (k+p+1)\widehat{\phi}(2k+p+2) = (s-k+1)\widehat{\phi}(2s-2k -p+2)\label{eq6}\\
p > s : & (k+s+1)\widehat{\psi}(2k+s+2) = (p-k+1)\widehat{\psi}(2p-2k-s+2)\label{eq5}
\end{align}
	\item For $\max(0,s-p) \leq k < s$ :
\begin{align}	\label{eq7}
 (k+p+1)\widehat{\phi}(2k+p+2)\widehat{\psi}(2k+2p-s+2) = (s-k+1)\widehat{\phi}(p+2)\widehat{\psi}(s+2)
\end{align}
  \item For $\max(0,p-s) \leq k < p$
\begin{align} \label{eq8}
 (k+s+1)\widehat{\phi}(2k+2s-p+2)\widehat{\psi}(2k+s+2) = (p-k+1)\widehat{\phi}(p+2)\widehat{\psi}(s+2)
\end{align}
 \item For $ k \geq s$ :
\begin{equation}\label{eq9}
 (k+p+1)\widehat{\phi}(2k+p+2)\widehat{\psi}(2k+2p-s+2) = (k-s+1)\widehat{\phi}(2k-2s+p+2)\widehat{\psi}(2k-s+2)
\end{equation}
\item For $k \geq p$ :
\begin{equation}\label{eq10}
(k-p+1)\widehat{\phi}(2k-p+2)\widehat{\psi}(2k-2p+s+2) = (k+s+1)\widehat{\phi}(2k+2s-p+2)\widehat{\psi}(2k+s+2)
\end{equation}
\end{itemize}
\end{lem}

\section{Results}
In \cite{lz2}, the uniqueness of the commutant is proved when both quasihomogeneous Toeplitz operators are of positive degree. A similar result is obtained by the next theorem, when the degrees of the Toeplitz operators are of opposite sign.  
\begin{thm}
Let $p$, $s$ be two positive integers and $\psi$ be a non-constant radial function. If there exists a non-constant radial function $\phi$ such that $T_{e^{ip\theta}\phi}$ commutes with $T_{e^{-is\theta}\psi}$, then $\phi$ is unique up to a multiplicative constant.
\end{thm}

\begin{proof}
Assume there exist two non-constant radial functions $\phi_1$ and $\phi_2$ such that both $T_{e^{ip\theta}\phi_1}$ and $T_{e^{ip\theta}\phi_2}$ commute  with $T_{e^{-is\theta}\psi}$. Then Equation (\ref{eq9}) implies that for all $k \geq s$ :
\begin{equation}\label{eq11a}
(k+p+1)\widehat{\phi_1}(2k+p+2)\widehat{\psi}(2k+2p-s+2) = (k-s+1)\widehat{\phi_1}(2k-2s+p+2)\widehat{\psi}(2k-s+2),
\end{equation}
and
\begin{equation}\label{eq11b}
(k+p+1)\widehat{\phi_2}(2k+p+2)\widehat{\psi}(2k+2p-s+2)=(k-s+1)\widehat{\phi_2}(2k-2s+p+2)\widehat{\psi}(2k-s+2).
\end{equation}
From these equalities, we obtain
\begin{equation*}
\widehat{\phi_1}(2k+p+2)\widehat{\phi_2}(2k-2s+p+2) = \widehat{\phi_1}(2k-2s+p+2)\widehat{\phi_2}(2k+p+2)
\end{equation*}
for all $k$ in the set $Z=\{ k\geq s: \widehat{\psi}(2k+2p-s+2)\widehat{\psi}(2k-s+2)\neq 0\}$. Clearly, $\displaystyle{\sum_{k\in Z}\frac{1}{k}=\infty}$ because $\psi$ is not identically null. Now using Lemma \ref{blk2}, we complexify the equation above and we obtain
$$\widehat{\phi_1}(z+2s)\widehat{\phi_2}(z) = \widehat{\phi_1}(z)\widehat{\phi_2}(z+2s) \textrm{ for }\Re z > 0.$$
Therefore \cite[ Lemma 6~p.1468]{l} implies that $\phi_1 = c \phi_2$ for some constant $c$.
\end{proof}

In \cite{lz2}, it is established that nontrivial quasihomogeneous Toeplitz operators with quasihomogeneous degrees of opposite signs do not commute on the analytic Bergman space of the unit disk. Here we shall demonstrate that it is not the case for the analogous Toeplitz operators defined on $L^2_h$.

\begin{thm}\label{th3}
 Let $p$, $s$, $m$ and $n$ be four  integers such that $p \geq s>0$, $m\geq 0$ and $n=(2m+1)s$. If there exists a nonzero radial function $\phi$ such that $T_{e^{ip\theta}\phi}$ commutes with $T_{e^{-is\theta}r^n}$, then $p=s=1$ and $\displaystyle{\phi(r)=\sum_{j=0}^{m}c_j r^{2j-1}}$.
\end{thm}

\begin{proof}
If $T_{e^{ip\theta}\phi}$ commutes with $T_{e^{-is\theta}r^n}$, then
\begin{equation}\label{eq16}
T_{e^{ip\theta}\phi}T_{e^{-is\theta}r^n}(z^k)=T_{e^{-is\theta}r^n}T_{e^{ip\theta}\phi}(z^k), \textrm{ for all }k\geq 0,
\end{equation}
and
\begin{equation}\label{eq17}
 T_{e^{ip\theta}\phi}T_{e^{-is\theta}r^n}(\bar{z}^k)=T_{e^{-is\theta}r^n}T_{e^{ip\theta}\phi}(\bar{z}^k), \textrm{ for all }k\geq 0.
\end{equation}
By Lemma \ref{lemdz} and since $\widehat{r^n}(z)=\dfrac{1}{z+n}$, equation \eqref{eq16} implies that
\begin{eqnarray}
\frac{2k-2s+2}{2k-s+n+2}\widehat{\phi}(2k-2s+p+2)&=&\frac{2k+2p+2}{2k+2p-s+n+2}\widehat{\phi}(2k+p+2),\textrm{ if }k\geq s, \label{eq18}\\
\displaystyle{\frac{2s-2k+2}{s+n+2}}\widehat{\phi}(p+2)&=&\frac{2k+2p+2}{2k+2p-s+n+2}\widehat{\phi}(2k+p+2), \textrm{ if } k<s. \label{eq19}
\end{eqnarray}
On the other hand, equation \eqref{eq17} implies that
\begin{eqnarray}
&\dfrac{2k+2s+2}{2k+s+n+2}\widehat{\phi}(2k+2s-p+2)&=\frac{2k-2p+2}{2k-2p+s+n+2}\widehat{\phi}(2k-p+2),\textrm{ if } k\geq p,\label{eq20}\\
&&\nonumber\\
&\dfrac{2k+2s+2}{2k+s+n+2}\widehat{\phi}(2k+2s-p+2)&=\frac{2p-2k+2}{s+n+2}\widehat{\phi}(p+2),\textrm{ if } p-s\leq k<p,\\
&&\nonumber\\
&\dfrac{2k+2s+2}{2k+s+n+2}&=\dfrac{2p-2k+2}{2p-2k-s+n+2}, \textrm{ if } 0\leq k<p-s. \label{eq22}
\end{eqnarray}
It is easy to see that equation \eqref{eq18} can be obtained from equation \eqref{eq20} by substituting $k$ by $k+p-s$. So we shall proceed using equation \eqref{eq20} to determine the form of the radial symbol $\phi$. By setting $z=2k-2p+2$, we complexify equation \eqref{eq20} and we obtain
$$\frac{z+2p+2s}{z+2p+s+n}\widehat{\phi}(z+p+2s)=\frac{z}{z+s+n}\widehat{\phi}(z+p), \textrm{ for }\Re z>0.$$
Here, we notice that the function defined by
$$f(z)= \frac{z+2p+2s}{z+2p+s+n}\widehat{\phi}(z+p+2s)-\frac{z}{z+s+n}\widehat{\phi}(z+p)$$
is analytic and bounded in the right-half plane and vanishes at $z=2k-2p+2$ for any $k\geq p$. Hence, by Lemma \ref{blk1}, we have $f(z)\equiv 0$. Therefore, we obtain that in the right half-plane
\begin{equation}\label{eq23}
\frac{\widehat{r^p\phi}(z+2s)}{\widehat{r^p\phi}(z)}=\frac{z(z+2p+s+n)}{(z+s+n)(z+2p+2s)}, \textrm{ for } \Re z>0.
\end{equation}
Since $n=(2m+1)s$ and using the well-known identity $\Gamma(z+1)=z\Gamma(z)$, where $\Gamma$ is the Gamma function, equation \eqref{eq23} can be written as
\begin{equation}\label{eq21}
\frac{\widehat{r^p\phi}(z+2s)}{\widehat{r^p\phi}(z)}=\frac{F(z+2s)}{F(z)}\textrm{ for } \Re z>0,
\end{equation}
where $\displaystyle{F(z)=\frac{\Gamma(\frac{z}{2s})\Gamma(\frac{z}{2s}+\frac{p}{s}+m+1)}{\Gamma(\frac{z}{2s}+m+1)\Gamma(\frac{z}{2s}+\frac{p}{s}+1)}}$. Next, equation \eqref{eq21}, combined with \cite[Lemma 6~p. 1468]{l}, implies there exists a constant $C$ such that
\begin{equation}\label{eq25}
\widehat{r^p\phi}(z)=CF(z), \textrm{ for }\Re z>0.
\end{equation}
Now, we shall show that $F(z)$ is the Mellin transform of a bounded function. Using the well-known property of the Gamma function namely
$$\Gamma(z+n)=(z+n-1)(z+n-2)\ldots z\Gamma(z)\textrm{ for }n\in\mathbb{N},$$ 
and after simplification, we obtain that
$$F(z)=\frac{(\frac{z}{2s}+\frac{p}{s}+m-1)(\frac{z}{2s}+\frac{p}{s}+m-2)\ldots(\frac{z}{2s}+\frac{p}{s}+1)}{(\frac{z}{2s}+m)(\frac{z}{2s}+m-1)\ldots\frac{z}{2s}},$$
which is a proper fraction in $z$ and can be written as sum of partial fractions
$$F(z)=\sum_{j=0}^{m}\frac{a_j}{z+2js}=\sum_{j=0}^{m}a_j\widehat{r^{2js}}(z).$$
Hence, equation \eqref{eq25} and Remark \ref{blk2} imply that
$$\phi(r)=\sum_{j=0}^{m}c_jr^{2js-p}.$$
At this point, we observe that the first term $r^{-p}$ in the expression of $\phi(r)$ is in $L^1(\mathbb{D},dA)$, and therefore $T_{e^{ip\theta}\phi}$ is a bounded Toeplitz operator, if and only if $p=1$.  Since by hypothesis we assumed $p\geq s>0$, we must have $p=s=1$. Finally, it is easy to verify that with $p=s=1$ equation \eqref{eq22} is satisfied, and also that the function $\displaystyle\phi(r)=\sum_{j=0}^{m}a_jr^{2j-1}$ satisfies equations \eqref{eq19} and \eqref{eq21}. 
\end{proof}

\begin{remark}\label{analytic}
It is well known that on the analytic Bergman space of the unit disk, non-trivial anti-analytic Toeplitz operators (resp. analytic Toeplitz operators) commute only with other such operators. Theorem 2 tells us that it is not the case in $L^2_h$. In fact if we take $m=0$, we see that $T_{\bar{z}}$ commutes with $T_{\phi}$ where $\displaystyle{\phi(z)=\frac{1}{\bar{z}}}$. Moreover, the symbol $\phi$ is obviously not bounded, but it is the so-called "nearly bounded symbol" \cite[p~204]{az}. Hence $T_{\phi}$ is bounded.
\end{remark}
The next theorem is a slight generalization of the previous one.
\begin{thm}
Let $p$ be a positive integer and $\phi$, $\psi$ be nonzero radial functions. If $T_{e^{ip\theta}\phi}$ commutes with $T_{e^{-ip\theta}\psi}$, then $p$ is equal to 1 and $\phi$, $\psi$ satisfy the following Mellin convolution equation
$$ \phi\ast_M\psi=C\left(\frac{1}{r}-r\right), $$
where $C$ is a constant.
\end{thm}

\begin{proof}
If $T_{e^{ip\theta}\phi}$ commutes with $T_{e^{-ip\theta}\psi}$, then the equalities in Lemma \ref{lm2} (with $p = s$) are reduced to two main equations:
\begin{itemize}
\item for all $0 \leq k < p$, 
\begin{equation}\label{eq12}
(k+p+1)\widehat{\phi}(2k+p+2)\widehat{\psi}(2k+p+2) = (p-k+1)\widehat{\phi}(p+2)\widehat{\psi}(p+2);
\end{equation}
\item for all $ k \geq p$,
\begin{equation}\label{eq13}
(k-p+1)\widehat{\phi}(2k-p+2)\widehat{\psi}(2k-p+2)= (k+p+1)\widehat{\phi}(2k+p+2)\widehat{\psi}(2k+p+2).
\end{equation}
\end{itemize}
We complexify  Equation $\eqref{eq13}$ by letting $z = 2k-2p+2$  and we obtain
\begin{equation}\label{eq14}
z\widehat{\phi}(z+p)\widehat{\psi}(z+p) = (z+4p)\widehat{\phi}(z+3p)\widehat{\psi}(z+3p), \quad \text{ for } \Re(z) > 0.
\end{equation}
Using the multiplicative property of the Mellin convolution, Equation \eqref{eq14} is equivalent to
\begin{equation}\label{M}
\frac{\mathcal{M}\left(r^p\phi \ast_M r^p\psi\right)(z+2p)}{\mathcal{M}\left(r^p\phi\ast_M r^p\psi\right)(z)} = \frac{F(z+2p)}{F(z)},
\end{equation}
where $F$ the function defined by
\begin{equation*}
F(z)=\frac{\Gamma\left(\frac{z}{2p}\right)}{\Gamma\left(\frac{z}{2p}+2\right)}, 
\end{equation*}
and $\Gamma$ is the Gamma function. Using the well-known identity $\Gamma(z+1)=z\Gamma(z)$, we can write
\begin{equation*}
F(z)=\frac{1}{z(z+2p)}= \frac{1}{z}-\frac{1}{z+2p} = \widehat{(\Un-r^{2p})}(z), 
\end{equation*}
where $\Un$ denote the constant function with value one. Now,  \cite[Lemma 6~p.1468]{l} combined with Equation \ref{M} implies the existence of a constant $C$ such that  
$$\mathcal{M}(r^p\phi\ast_M r^p\psi)(z)=C F(z)\textrm{  for } \Re z>0.$$ 
Therefore
$$\phi \ast_M \psi = C(\frac{1}{r^p}-r^p).$$
Finally, it is easy to see that this Mellin convolution (as a function) is in $L^1([0,1], rdr)$  if and only if $p=1$, and that in this case Equation \eqref{eq12} is satisfies.
\end{proof}

\begin{remark}
As an example one can show that, if $\alpha \geq -1$ is an integer and if  
$T_{e^{-i\theta}r^\alpha}$ commutes with $T_{e^{i\theta}\phi}$, then  
\begin{equation*}
\phi(r) = \frac{1}{2}\left(\frac{\alpha+1}{r}-(\alpha-1)r \right).
\end{equation*}
In particular, if $\alpha=1$ we obtain the example of  Remark \ref{analytic}.
\end{remark}

Now we shall consider the case when $p < s$.
\begin{thm}\label{th4}
Let $p, s, m,$ and $n$ be  integers such that $0<p < s$, $m \geq 0$ and $n=(2m+1)s$. Assume there exists a radial function $\phi$ such  that $T_{e^{ip\theta}\phi}$ commutes with $T_{e^{-is\theta}r^n}$. Then $p = 1$ and 
\begin{itemize}
 \item $\displaystyle \phi(r)=\sum_{j=0}^{m}c_jr^{2j-1}$ if $s\leq m+1$,
 \item $\phi\equiv 0$ if $s> m+1$.
 \end{itemize}
\end{thm}

\begin{proof}
Assume that $T_{e^{ip\theta}\phi}$ commutes with $T_{e^{-is\theta}r^n}$. Then, by a similar argument as in the proof of Theorem \ref{th3}, we show that $p=1$  and $\displaystyle \phi(r)=\sum_{j=0}^{m}c_jr^{2js-1}$. Moreover, Lemma \ref{lm2} implies that   
\begin{equation*}
(k+2)\widehat{\phi}(2k+3)=(s-k+1)\widehat{\phi}(2s-2k+1) \text{ for all } 0 \leq k \leq s-1,
\end{equation*}
which is equivalent to
\begin{equation}\label{eq28}
(k+2)\sum_{j=0}^{m}\frac{c_j}{2k+2js+2}=(s-k+1)\sum_{j=0}^{m}\frac{c_j}{2s-2k+2js}, \textrm{ for all } 0\leq k\leq s-1.
\end{equation}
Now, the $s$ equalities given by Equations \eqref{eq28} can be written as a homogeneous linear system in the following way
\begin{equation}\label{S}
A\left(\begin{array}{c}c_0\\  \vdots\\ c_m \end{array}\right)=\left(\begin{array}{c} 0\\ \vdots\\ 0 \end{array}\right),
\end{equation}
where the matrix $A$ is of size $s\times (m+1)$ and its entries are given by
$$a_{kj}=\frac{k+2}{2k+2js+2}-\frac{s-k+1}{2s-2k+2js},$$
for $0\leq k\leq s-1$ and $0\leq j\leq m$. Clearly the $s$ rows of the matrix $A$ are linearly independent. Thus
\begin{itemize}
 \item [$\bullet$] if $s\leq m+1$, the homogeneous system (\ref{S}) has less equations than unknowns and therefore it has a nontrivial solution i.e., the $c_j$ are not all equal to zero.
 \item [$\bullet$] if $s> m+1$, (\ref{S})  has more independent equations than unknowns and hence $c_j = 0$ for all $0\leq j\leq m$  i.e., $\phi\equiv 0$.
\end{itemize}
\end{proof}

\noindent
\textbf{Acknowledgments.}The second author is grateful to Professor Michel Rajoelina and Dr Elizabeth Strouse, his PhD advisor, for their useful suggestions.

\end{document}